\documentclass[11pt,a4paper]{article}
\usepackage[OT1]{fontenc}
\usepackage[all]{xy}

\usepackage[english]{babel}
\usepackage{amsmath}
\usepackage{amsfonts}
\usepackage{amsthm}
\def\d{ {\cal D} }

\def\a{ {\cal A} }
\def\h{ {\cal H} }

\def\b{ {\cal B} }
\def\u{ {\cal U} }
\def\ii{ {\cal I} }

\def\s{ {\cal S} }
\def\e{ {\cal E} }
\def\p{ {\cal P} }

\def\k{ {\cal K} }
\def\f{ {\cal F} }

\def\m{ {\cal M} }
\def\pdk{\p_{\d+\k}}

\def\xx{ {\bf x} }
\def\yy{ {\bf y} }
\def\zz{ {\bf z} }
\def\cc{ {\bf c} }
\def\pp{ {\bf p} }
\def\ss{ {\bf s} }
\def\ww{ {\bf w} }
\def\ccc{ {\bf c_0} }
\def\dd{ {\bf d} }
\def\tt{ {\bf t} }
\def\ee{ {\bf e} }
\def\ff{ {\bf f} }

\def\noi{\noindent}

\def\g1{ \mathfrak{g}_1  }

\newtheorem{teo}{Theorem}[section]
\newtheorem{prop}[teo]{Proposition}
\newtheorem{lem}[teo]{Lemma}
\newtheorem{coro}[teo]{Corollary}

\theoremstyle{definition}
\newtheorem{rem}[teo]{Remark}

\newtheorem{quest}[teo]{Question}

\title{The C$^*$-algebra of compact perturbations of diagonal operators}
\author{E. Andruchow, E. Chiumiento, A. Varela}

\begin{document}

\maketitle 

\begin{abstract}
We study the elementary C$^*$-algebra $\d+\k$ which consists of sums of a 
diagonal plus a compact operator. We describe the structure of the unitary 
group, the sets of ideals, automorhisms and projections.
  \end{abstract}
\bigskip


\section{Introduction}
The motive of this paper is the  elementary C$^*$-algebra $\d+\k$ consisting of operators in a separable Hilbert space $\h$ (with a fixed or canonical orthonormal basis) which are sums of diagonal and compact operators. It is an algebra with plenty of ideals (even maximal ideals: as many as there are points in the residual set $\beta\mathbb{N}\setminus \mathbb{N}$, where  $\beta\mathbb{N}$  is  the Stone-Cech compactification of the natural numbers), projections and homomorphisms. Many of the properties discussed here can be studied using general or abstract techniques. However,  we prefer, when possible, to present  elementary or direct proofs. We focus on the set of projections and its geometry, on the structure of the unitary group, and on the automorphisms of $\d+\k$.

Section 2 contains preliminary facts on $\d+\k$ (e.g. $\d+\k$ is nuclear, quasi-diagonal, evidently contains the ideal $\k$ of compact operators, with quotient $\ell^\infty/\cc_0$, etc). In Section 3 we examine the general properties of the  groups of invertibles and unitaries. For instance a unitary $U$ element in $\d+\k$  can be factorized $U=e^{iX}D$, where $D$ is unitary and diagonal, and $X^*=X$ is compact. In Section 4 we consider ideals, characters and positive functionals. It is shown that all automorphims are approximately inner (i.e., implemented by unitaries in $\h$). Two questions remain un-answered:
\begin{itemize}
\item
Are all unitary elements of $\d+\k$ exponentials (with exponent in $\d+\k$)?
\item
Are the unitary operators of $\h$ which implement the automorphisms of $\d+\k$ a product of a unitary in $\d+\k$ times a {\it permutation} unitary operator?
\end{itemize}
In Section 5 we study the set of projections. Our main result (Theorem 4.7) describes the structure of this set. The main tools in this description are the concept of index of a pair of projections \cite{ass} and the resticted Grassmanian associated to a given decomposition of $\h$ \cite{sato}, \cite{segalwilson}. In Section 6 we compute the first homotopy group of the unitary group of $\d+\k$, and the $K$-groups of the algebra.

\section{Preliminaries}\label{preliminar}
Let $\h$ be a separable Hilbert space. We denote by $\d=\d(\h)$ the algebra of diagonal operators with respect to a fixed orthonormal basis $\{ e_n: n \ge 1\}$, and by $\k=\k(\h)$ the ideal of compact operators. Our object of study is the algebra
$$
\d+\k:=\{D+K: D\in\d \hbox{ and } K\in\k\}.
$$
If $\xx=\{x_n\}_{n\ge 1}$ is a sequence in $\ell^\infty$, we shall denote by $D_\xx$ the diagonal operator whose entries are $x_n$.
Clearly, if $T\in\d+\k$, $T$ decomposes as a sum of a diagonal plus a compact operator in many ways. For instance, if $\yy\in \ccc$ (=sequences converging to zero), then $T=D+K=D+D_\yy+K-D_\yy$, are two such decompositions. Apparently, all decompostions arise in this form from a given one. There is though one  distinguished decomposition. Denote by 
$$
\Delta: \b(\h)\to \d\subset \b(\h)
$$
the conditional expectation given by $\Delta(T)=D_\tt$, where $\tt=\{<Te_n,e_n>\}_{n\ge 1}$ (i.e., putting zeros in the off-diagonal entries of $T$). It is well known that $\Delta$ preserves the Schatten-von Neumann ideals, and is contractive for the $p$-norms. In particular, it preserves $\k$, i.e., $\Delta(\k)\subset \k$ and $\|\Delta(T)\|\le \|T\|$. The distinguished decomposition alluded above is
$$
T=D_T+K_T, \ \hbox{ with } \Delta(K_T)=0.
$$ 
The following result could be obtained as a particular case of a more general situation. However, it is an easy consequence of  the continuity of $\Delta$.
\begin{prop}
$\d+\k$ is closed in $\b(\h)$.
\end{prop}
\begin{proof}
Suppose that $T_n\in\d+\k$ and $T_n\to T$. Then $(I-\Delta)(T_n)\to (I-\Delta)(T)$. Note that if $T_n=D_n+K_n$, then 
$(I-\Delta)(T_n)=(I-\Delta)(K_n)\in\k$ ($I-\Delta$ also preserves $\k$). Then $T-\Delta(T)$ is compact. Then
$$
T=\Delta(T)+(T-\Delta(T))\in\d+\k.
$$
\end{proof}
Thus $\d+\k\subset \b(\h)$ is a C$^*$-algebra.
\begin{rem}\label{remark preliminar} Basic properties of $\d+\k$:


\begin{enumerate}
\item
$\d+\k$ is non separable: $\d+\k$ contains a copy of $\ell^\infty\simeq\d$. In particular, it is non AFD.
\item
$\d+\k$ is quasi-diagonal (see \cite{brown}).
\item
$\d+\k$ is nuclear. Indeed, $\k$ is nuclear and the quotient $(\d+\k)/\k$ is commutative, thus nuclear (see \cite{blackadar}, p. 369).
\item
$\d+\k$ is strong and weak operator dense in $\b(\h)$. 
\item
Any normal operator $A\in\b(\h)$ is unitarily equivalent to an element in $\d+\k$ (this is the classical Weyl-von Neumann-Berg theorem).
\item
The inclusion $\d+\k\subset\b(\h)$ is irreducible.
\item
If $T\in\d+\k$ is selfadjoint, then $D_T$ and $K_T$ are selfadjoint. Indeed, $\Delta(T)^*=\Delta(T^*)=\Delta(T)$.
\end{enumerate}
\end{rem}

Let us denote by
$\pi$ both quotient homomorphisms $\b(\h)\to \b(\h) /\k$ and its restriction $\d+\k \to (\d+\k) / \k$. 
It is known that $\pi(\d)\subset\b(\h)/\k$ is a maximal abelian subalgebra (shortly, a {\it masa}), called the {\it standard masa} of the Calkin algebra \cite{jp}. Then, also the following characteristic property of $\d+\k$ can be obtained.
%
\begin{prop}
Let $S\in\b(\h)$. Then the following properties are equivalent:
\begin{enumerate}
	\item[a) ] $S\in\d+\k$, 
	\item[b) ]  $[S,D]\in\k$ for every diagonal operator $D$,
	\item[c) ]  $[S,T]\in\k$ for all $T\in\d+\k$,
	\item[d) ]  $[S,P ]\in \k$ for every orthogonal projection $P\in \d$.
\end{enumerate}   
\end{prop}
\begin{proof}
Consider first a) $\Rightarrow$ b). If $S=D_0+K_0\in\d+\k$, then 
$[S,D]=[K_0,D]\in\k$. Conversely, suppose that $[S,D]\in\k$ for all $D\in\d$. 
Then 
$$
0=\pi([S,D])=[\pi(S),\pi(D)]
$$
for all $D\in\d$. Since $\pi(\d)$ is maximal abelian, this implies that 
$\pi(S)\in\pi(\d)$, i.e., $S\in\d+\k$.	
	
The assertion b) $\Leftrightarrow$ c) is clearly a trivial consequence of 
a) $\Leftrightarrow$ c).  

b) $\Rightarrow$ d) is trivial. Conversely, suppose that $S\in B(H)$ satisfies 
$[S,P ]\in \k$ for every orthogonal projection $P\in \d$. We can approximate 
every diagonal $D\in\d$ with diagonals of finite spectrum as sequences of 
$\ell^\infty$ can be approximated by those that take finite values. Therefore, 
for 
 $\varepsilon>0$ there exists $D_\varepsilon\in\d$ of the form
$$
D_\varepsilon=\sum_{k =1}^{n} \alpha_{k } P_{k }, \text{ with } \alpha_{k }\in 
\mathbb{C} 
\text{ and } P_{k } \text{  orthogonal projections in } \d 
$$ 
such that $\|D-D_\varepsilon\|<\varepsilon$.
Observe that $S$ satisfies $[S,P_k]\in \k$ for $k=1,\dots, n$. Then it follows that
$$
\|[\pi(D),\pi(S)]\|=\|\pi(D) \pi(S) -\pi(S) \pi(D)\|= \|\pi(D) \pi(S) -\pi(S) 
\pi(D)+\pi(S)\pi(D_\varepsilon)-\pi(D_\varepsilon)\pi(S)\| 
$$
because $[\pi(S),\pi(D_\varepsilon)]\in\k$. Moreover,
\begin{equation*}
\begin{split}
\|\pi(D) \pi(S) -\pi(S)\pi(D) & + 
\pi(S)\pi(D_\varepsilon)-\pi(D_\varepsilon)\pi(S)\|=\\
& = 
\|(\pi(D)-\pi(D_\varepsilon)) \pi(S) + \pi(S)( 
\pi(D_\varepsilon)-\pi(D))\|\\
& \leq \|D-D_\varepsilon\|\, \|S\|+\|S\|\, \|D_\varepsilon-D\|
 < 2\ \|S\|\ \varepsilon
\end{split}
\end{equation*}
for every $\varepsilon>0$. 
This implies that $\|[\pi(D),\pi(S)]\|=0$, which completes the proof.
\end{proof}
Let us finish this preliminary section by showing that elements of $\d+\k$ with finite spectrum are norm dense in $\d+\k$.

\begin{prop}
The set of elements with finite spectrum is norm dense in $\d+\k$. Also the set of selfadjoint elements with finite spectrum is norm  dense in the set of selfadjoint elements of $\d+\k$, i.e.,  $\d+\k$ has real rank zero.
\end{prop}
\begin{proof}
Fix $\epsilon >0$. Denote by $E_k$ the rank one orthogonal projection onto the line generated by the vector $e_k$ of the fixed basis, and by $F_n=\sum_{k=1}^nE_k$. Given $D+K\in\d+\k$, let $K_n=F_nKF_n$ and $D_n=F_nD=DF_n$. Clearly $K_n\to K$, and thus $\|K-K_n\|<\epsilon/2$ for $n$  sufficiently large. Denote by $D'_n=F_n^\perp D$, and note that $D'_n$ is a diagonal operator acting on $R(F_n)^\perp$ (in terms of the basis $\{e_k: k\ge n+1\}$). Sequences with finite many values are dense in $\ell^\infty$, thus there exists a diagonal operator $D_0$, acting in  $R(F_n)^\perp$, such that $\|D'_n-D_0\|<\epsilon/2$. Consider the element $D_n+D_0+K_n$. Clearly, it belongs to $\d+\k$. Note that in terms of the decomposition $R(F_n)\oplus R(F_n)^\perp$, it can be written
$$
(D_n+K_n) \oplus D_0,
$$
and thus $\sigma(D_n+D_0+K_n)=\sigma(D_n+K_n)\cup\sigma(D_0)$, and clearly both sets are finite (the left hand side set is the spectrum of an operator in the finite dimensional space $R(F_n)$). Finally, noting that $D-D_n=D'_n$,
$$
\|D+K-(D_n+D_0+K_n)\|\le \|K-K_n\|+ \|D'_n-D_0\|<\epsilon.
$$
If we start with a selfadjoint element $D+K$ (with $D^*=D$ and $K^*=K$), it is clear that the finite spectrum approximant is also selfadjoint.
\end{proof}

\section{The linear group of $\d + \k$}

Let us denote by $G_{\d+\k}$ the {\it linear group} (or group of invertible elements) of $\d+\k$.

\begin{lem}\label{lemita}
If $T\in\d+\k$ is invertible, then there exists a decomposition $T=D_0+K_0$ with $D_0$ invertible.
\end{lem}
\begin{proof}
Let $T=D+K$, and denote by $\dd=\{d_n\}_{n\ge 1}$ the sequence of the entries of $D$. We claim that $\dd$ cannot have a subsequence which converges to $0$. Suppose otherwise that there exists a subsequence $d_{n_k}\to 0$. Consider the sequence $\zz=\{z_n\}_{n\ge 1}$ given by $z_n=1$ if $n=n_k$, and $0$ if not. Then $D_\zz$ is a projection such that $D_\zz D$ is compact. Then $D_\zz T=D_\zz D+D_\zz K$ is compact. Since $T$ is invertible, this implies that 
$D_\zz=(D_\zz T)T^{-1}$ is a compact projection. Then it must have finite rank, which leads to a contradiction. It follows that there exists $r>0$ such that $|d_n|\ge r$ for all $n\ge 1$, save for a finite set $\{n_1, \dots, n_N\}$. Let $P$ be the diagonal operator with $1$ in the place $n_j$, $1\le j\le N$, and zero elsewhere. Clearly $P$ is compact, and $D_0=D+r P\in\d$ is invertible, for a suitable choice of $r$. Then $T=D_0+(K-r P)$, with $K-r P\in\k$ . 
\end{proof}

\begin{prop}
Let $T\in\d+\k$. The following are equivalent:
\begin{enumerate}

\item
$T$ is invertible.
\item
$N(T)=\{0\}$ and there exists a decomposition $T=D+K$ with $D$ invertible.
\item
$R(T)=\h$ and there exists a decomposition $T=D+K$ with $D$ invertible.
\end{enumerate}
\end{prop}

\begin{proof}
By the previous lemma, 1) implies 2) and 3). Conversely, if  $T=D+K$ with $D\in\d$ invertible,  then $T=D(I+D^{-1}K)$.  Thus, if $I+D^{-1}K$ has either trivial nullspace or full range,  by Fredholm's alternative, $I+D^{-1}K$ is invertible. 
\end{proof}
Denote by $\b_s(\h)$ the set of self-adjoint operators in $\b(\h)$. 
\begin{prop}
$G_{\d+\k}$ is dense in $\d+\k$ (i.e., $\d+\k$ has stable rank one {\rm \cite{gustavo pucho}}). The same holds for the selfadjoint parts: $G_{\d+\k}\cap \b_s(\h)$ is dense in $(\d+\k)\cap\b_s(\h)$. 
\end{prop}
\begin{proof}
Let $T=D+K\in\d+\k$, with $D$ given by the sequence $\{d_n\}_{n\ge 1}$, and fix $\epsilon>0$. Let $D_\epsilon\in\d$ given by the sequence $\{d^\epsilon_n\}_{n\ge 1}$,
$$
d^\epsilon_n=\left\{ \begin{array}{l} d_n \hbox{ if }|d_n|\ge \epsilon \\ \epsilon \hbox{ if } |d_n|<\epsilon \end{array} \right. .
$$
Clearly, $D_\epsilon$ is invertible and $\|D-D_\epsilon\|\le 2\epsilon$. Let $M\in\k$ be a finite rank operator such that $\|K-M\|<\epsilon$. Then also $D_\epsilon^{-1}M$ has finite rank, and thus finite spectrum. We can further adjust $M$ in order that $-1$ does not belong to the spectrum of $D_\epsilon^{-1}M$ (for instance, adding a small multiple of a projection onto the finite rank subspace on which $D_\epsilon^{-1}M$ acts), and still have that $\|K-M\|<\epsilon$. Then  $D_\epsilon+M=D_\epsilon(I+D_\epsilon^{-1}M)$ is invertible, and
$$
\|D+K-(D_\epsilon+M)\|\le \|D-D_\epsilon\|+\|K-M\|<3\epsilon.
$$
If $T, D$ and $K$ are selfadjoint, then the operator $D_\epsilon$ is clearly selfadjoint. Also the finite rank operator $M$ approximating $K$ can be chosen selfadjoint.
\end{proof}

Next we consider the relation between the commutative $C^*$-algebra $(\d + \k) / \k$ and Fredholm operators.

\begin{prop}
Let $T \in \d + \k$, then the following assertions are equivalent:
\begin{enumerate}
\item[1.] $\pi(T)$ is invertible in $(\d + \k )/ \k$. 
\item[2.] $T$ is Fredholm.
\item[3.] $T$ is semi-Fredholm.
\end{enumerate}
In any of these cases, $ind(T)=0$.
\end{prop}
\begin{proof}
To prove that $2)$ implies $1)$, we first note that $T=D+K \in \d +\k$, is Fredholm if and only if $D$ is. Next we observe that for a diagonal operator we have $N(D)=R(D)^\perp$. Then $k:=\dim N(D)=\dim R(D)^\perp$ gives the   zero entries in the sequence $\{ d_n\}_{n \geq 1}$ which defines $D$. Suppose that we list all these zero entries $d_{i_1}, \ldots , d_{i_k}$. The diagonal operator $D_{\textbf{a}}$, where $\textbf{a}=\{ a_n\}_{n \geq 1}$ is defined as $a_n=d_n^{-1}$ if $n \neq i_1 , \ldots i_k$, and $a_n=0$ otherwise, clearly satisfies that $DD_{\textbf{a}} -I=D_{\textbf{a}}D-I$ is a finite rank  operator. Therefore $TD_{\textbf{a}} -I$ and $D_{\textbf{a}}T-I$ are compact operators, which means that $\pi(T)$ is invertible in $(\d+ \k)/\k$.  The reverse implication is trivial. 

The equivalence between $2)$ and $3)$ follows again by noting that $T=D+K \in \d +\k$, is Fredholm if and only if $D$ is, and also that  $N(D)=R(D)^\perp$. In addition, this also implies that $ind(T)=0$.
\end{proof}


The following elementary consequence follows: 

\begin{coro}
There are no proper isometries in $\d+\k$: if $V\in\d+\k$ is isometric, then $V$ is a unitary operator. Similarly, there are no proper co-isometries in this algebra.
\end{coro}
\begin{proof}
An isometry $V$  is a semi-Fredholm operator with $-\infty\le ind(V)\le 0$. Thus $ind(V)=0$. Since $N(V)=\{0\}$, it follows that $R(V)^\perp=\{0\}$, i.e. $R(V)=\overline{R(V)}=\h$.
\end{proof} 

Let us describe unitary elements in $\d+\k$.
\begin{prop}\label{factores unitarios}
If $U\in\d+\k$ is unitary, then there exists a decomposition $U=D+K$ with $D$ unitary  and $K$ compact. Moreover, $U$ can be factorized as
$$
U=De^{iX}
$$
with $X=X^*\in\k$, $\|X\|\le \pi$.
\end{prop}
\begin{proof}
Let  $U=D'+K'$ be an arbitrary decomposition of $U$. Since $U$ is unitary, $I=U^*U=D'^*D'+D'^*K'+K'^*D'+K'^*K'$ and thus $D'^*D'=I+K''$ with $K''\in\k$. It follows that the spectrum of $D'^*D'$ is countable, and accumulates only eventually at $1$, i.e., if $d_n'$ are the entries of $D'$, then $|d'_n|\to 1$. Let $d'_n=|d'_n|e^{i\theta_n}$, with $-\pi\le \theta_n<\pi$. Put $D=D_\dd$, where $\dd=\{d_n\}_{n\ge 1}$, and $d_n=e^{i\theta_n}$. Then $D'-D$, which is the diagonal operator given by the sequence $(|d'_n|-1)e^{i\theta_n}$ (which converges to zero), is compact. Then  $K=K'+D'-D\in \k$ and $U=D+K$ with $D$ unitary.

Then $U=D(I+D^*K)$. The operator $I+D^*K$ is unitary. Thus $D^*K$ is normal and compact: there exist mutually orthogonal selfadjoint projections $P_n$ of finite rank, such that
$$
D^*K=\sum_{n\ge 1} \lambda_n P_n.
$$
Put $P_0=I-\sum_{n\ge 1}P_n$. Then 
$$
1+D^*K=P_0+\sum_{n\ge 1} (1+\lambda_n)P_n.
$$
Since $I+D^*K$ is unitary, $|1+\lambda_n|=1$. Let $-\pi\le\chi_n<\pi$ such that $e^{i \chi_n}=1+\lambda_n$, and put
$$
X=\sum_{n\ge 1} \chi_n P_n.
$$
The fact that $1+\lambda_n$ accumulate only eventually at $1$, implies that $\{ \chi_n \}_{n \geq 1}$ accumulate only eventually at $0$. Thus $X$ is compact. Clearly $\|X\|\le\pi$ and $e^{iX}=I+D^*K$.
\end{proof}
The unitary group $\u_{\d+\k}$ of $\d+\k$ was studied in \cite{bv}. The above factorization was found there.
The following is a straightforward consequence:
\begin{coro}
$\u_{\d+\k}$ and $G_{\d+\k}$ are connected.
\end{coro}
\begin{proof}
By the above factorization, any element $U\in\u_{\d+\k}$ can be factorized $U=De^{iX}$, with $D=D_\dd$, $\dd=\{e^{i\theta_n}\}_{n\ge 1}$. Then $U(t)=D(t)e^{itX}$, with $D(t)=D_{\dd(t)}$, $\dd(t)=\{e^{it\theta_n}\}_{n\ge 1}$ is a continuous path of unitary elements in $\u_{\d+\k}$, such that $U(0)=I$ and $U(1)=U$.

Let $G\in G_{\d+\k}$, and $G=U|G|$ its polar decomposition, which remains in $\d+\k$. The set of positive invertible elements  in a C$^*$-algebra is convex, thus connected. As above, $U$ can be connected with $I$ by a continuous path. 
\end{proof}

\begin{quest}
One question that we consider interesting is whether the unitary group of $\d+\k$ is exponential, i.e., whether $\u_{\d+\k}$ equals $\{e^{i(D+K)}: D^*=D, K^*=K\}$. The {\it exponential rank} measures how many exponentials are needed to factorize any unitary element in the connected component of the identity. Clearly, from the above Proposition, this number is less or equal than $2$. H.X. Lin \cite{lin} proved that unital C$^*$-algebras with real rank zero (as $\d+\k$) have this invariant less or equal than $1+\epsilon$, meaning that any unitary element (in the connected component of the identity) is a limit of exponentials.
\end{quest}   

$\d+\k$ has the FU property (\cite{phillips}, Proposition 1.5): unitary elements with finite spectrum are norm dense in $\u_{\d+\k}$.

Let us complete this section with a factorization result for positive invertible elements. It is a particular case of a remarkable  general result by H. Porta and L. Recht \cite{prfactorizacion}. Let us transcribe their result:
\begin{teo}{\rm (Porta and Recht \cite{prfactorizacion}, Corollary 7)}
Let $\b\subset\a$ be unital C$^*$-algebras and $E:\a\to\b$ a conditional expectation. Then any positive and invertible element $a\in\a$ has a unique factorization
$$
a=b^{1/2}e^z b^{1/2},
$$
where $b\in\b$ (is positive and invertible) and $z^*=z$ with $E(z)=0$.
\end{teo} 
This result is non trivial even for the case when $\a$ is the algebra of $3\times 3$ complex matrices and $\b$ the subalgebra of diagonal matrices.
In our case, where the conditional expectation is $\Delta:\d+\k\to\d$, this results yields the following:
\begin{coro}
Let $A\in\d+\k$ be positive and invertible. Then there exist a unique positive and invertible diagonal operator $D$, and a selfadjoint compact operator $Z$, with zero diagonal, such that 
$$
A=D^{1/2}e^ZD^{1/2}.
$$ 
\end{coro}
\begin{proof}
Note that if $Z\in\d+\k$ belongs to the kernel of $\Delta$, then $Z$ is compact:  recall the distinguished decomposition $Z=D_Z+K_Z$, 
$$
0=\Delta(D_Z)+\Delta(K_Z)=D_Z.
$$
\end{proof}
\section{Ideals, characters and positive functionals}

The quotient map $\pi$ onto the Calkin algebra restricted to $\d+\k$ gives a $*$-epimorphism
$$
\pi=\pi|_{\d+\k}: \d+\k\to (\d+\k) /\k.
$$

\begin{rem}
Recall that $(\d+\k)/\k$ is a commutative algebra. Denote by  $\ccc$ the Banach space of sequences which converge to $0$. Then, clearly,
$$
(\d+\k)/\k\simeq \ell^\infty/\ccc.
$$
Indeed, if $T\in\d+\k$, and $D$ is given by the sequence $\dd=\{d_n\}_{n\ge 1}$, the isomorphism is given by $\pi(T)\mapsto [\dd]$. 
The maximal ideal space of $\ell^\infty/\ccc$ consists of maximal ideals of $\ell^\infty$ which contain the ideal $\ccc$. The maximal ideal space of $\ell^\infty$ is the Stone-Cech compactification $\beta\mathbb{N}$ of the natural numbers $\mathbb{N}$. The ideals which contain $\ccc$ are those on the residual set, i.e. $\beta\mathbb{N}\setminus \mathbb{N}$.
\end{rem}
\begin{coro}
There are uncountable many different maximal ideals in $\d+\k$, which contain $\k$. In fact, there are $2^c$ maximal ideals.
\end{coro}
\begin{proof}
For any maximal ideal $\m$ in $\ell^\infty/\ccc$, let $\psi_\m:\ell^\infty/\ccc\to\mathbb{C}$ be the multiplicative functional such that $\ker \psi_\m=\m$. Then $\Psi_\m:=\psi_\m\circ\pi:\k+\d\to \mathbb{C}$ is a multiplicative functional. Clearly different maximal ideals $\m$ of $\ell^\infty/\ccc$ give rise to different characters $\Psi_\m$ of $\k+\d$, all of which vanish at $\k$. 
\end{proof}

Let us take a brief look at the positive functionals in $\d+\k$. There is an explicit way to decompose any $\varphi\ge 0$ in its  atomic and singular parts. Namely, given  $\varphi\ge 0$, it is clear that the restriction $\varphi_\k:=\varphi|_\k$ is a positive functional (eventually $\varphi_\k=0$) in $\k$. Then there exists a trace class operator $A_\varphi\ge 0$ such that 
$$
\varphi_\k(K)=Tr(A_\varphi K)
$$
for all   $K\in\k$. Also it is clear that $Tr(A_\varphi)=\|\varphi|_\k\| \le \|\varphi\|$.
Note that $\varphi_\k$ can be extended to $\d+\k$  naturally: $\varphi_\k(T)=Tr(A_\varphi T)$ (in fact, it can be extended to $\b(\h)$), and denote still by $\varphi_\k$ this extension. Put 
$$
\varphi_\infty:=\varphi-\varphi_\k.
$$
Then 
\begin{prop}
$\varphi_\infty\ge 0$, and $\varphi_\infty|_\k =0$.
\end{prop}
\begin{proof}
The second assertion is clear.
Let $T\ge 0$ in $\d+\k$, $T=D_T+K_T$ (unique decomposition with $E(K_T)=0$). Note that $D_T=\Delta(T)\ge 0$. 
Then 
$$
\varphi(T)=\varphi(D_T)+\varphi(K_T)=\varphi(D_T)+Tr(A_\varphi K_T),
$$
whereas $\varphi_\k(T)=Tr(A_\varphi D_T)+Tr(A_\varphi K_T)$. Thus, in order to check that $\varphi_\infty\ge 0$ we must show that
$$
Tr(A_\varphi D_T)\le \varphi(D_T).
$$
Let $D_n$ be an increasing sequence of finite rank diagonal operators such that $D_n\to D_T$ strongly (for instance, $D_n$ be the $n$-truncation of $D_T$). Then, since  $\varphi_\k$ is strongly continuous on bounded sets, 
$$
\varphi_\k(D_n)\to \varphi_\k(D_T).
$$
On the other hand, since $D_T-D_n\ge 0$,  and $D_n$ are compact,
$$
\varphi_\k(D_n)=\varphi(D_n)\le \varphi(D_T),
$$
which finishes the proof.
\end{proof}
Thus we have the decomposition $\varphi=\varphi_\k+\varphi_\infty$, with $\varphi_\k$ normal (i.e., strongly continuous on bounded sets) and $\varphi_\infty$ singular (i.e., vanishing on compact elements of $\d+\k$). These singular positive functionals are in one to one correspondence with positive functionals in $\ell^\infty / {\bf c_0}$.

The following result, shows, in particular, that the maximal ideals of $\ell^\infty/\cc_0$ induce all the maximal ideals of $\d+\k$.

\begin{prop}
Let $\ii\ne 0$ be an  ideal of $\d+\k$. Then $\k\subset\ii$.
\end{prop}
\begin{proof}
Clearly, it suffices to show that there is a non zero compact element in $\ii$. Let $0\ne D_0+K_0\in \ii$, $D_0=D_\dd$ and $E_j$ stands for the diagonal projection with $1$ in the $j$-coordinate. Note that if $E_j(D_0+K_0)\neq 0$, then we are done. So we assume that $E_j(D_0+K_0)\neq 0$ for all $j \geq 1$. Therefore the diagonal $\dd=\{ d_j \}_{j \geq 1}$ satisfies $d_j=-(K_0)_{jj}$, $j \geq 1$, which implies that $D_0$ is compact, and then $D_0+K_0$ also is compact.
\end{proof}

In particular, this allows us to describe the irreducible representations of $\d+\k$:
\begin{coro}
The irreducible representations of $\d+\k$ are the multiplicative functionals of $\d+\k$ (which necessarily vanish at $\k$,i.e., which are given by  characters of $\ell^\infty/{\bf c_0}$), or they are unitarily equivalent to the inclusion representation $\d+\k\subset \b(\h)$.
\end{coro}
\begin{proof}
By the above proposition, one has that the pure states of are either of the form $\varphi_\k=Tr(A\ \cdot)$ or $\varphi_\infty$. In the first case, since $\varphi$ is pure, it is clear that $A$ is a rank one projection. The corresponding G.N.S. representation is (unitarily equivalent to) the inclusion representation. In the second case, $\varphi=\varphi_\infty$  is pure, thus the corresponding representation induces an irreducible representation of $\ell^\infty/{\bf c_0}$, which is one dimensional, i.e. $\varphi_\infty$ is a character.
\end{proof}

The following result follows.

\begin{lem}\label{en K}
Let $\theta$ be a $*$-automorphism of $\d+\k$. Then $\theta(\k)=\k$
\end{lem}
\begin{proof}
Since $\theta(\k)$ is a proper ideal of $\d+\k$, $\k\subset \theta(\k)$. Also, by the same argument, $\k\subset\theta^{-1}(\k)$. Then, $\theta(\k)\subset \theta \theta^{-1}(\k)=\k$.
\end{proof}
Then  a  $*$-automorphism $\theta$ of $\d+\k$ induces, by restriction, a $*$-automorphism $\theta|_{\k}$ of $\k$. These  are given by  conjugation with  unitary operators in $\b(\h)$. Thus, there exists a unitary operator $U\in\u(\h)$ such that
$$
\theta(K)=UKU^* \ , \ \ \hbox{ for all } \ K\in\k.
$$
On the other hand, if an automorphism $\theta$ leaves $\d$ invariant (i.e., $\theta(\d)\subset\d$),  it
must be $\theta(\d)=\d$. Indeed, since $\d\subset\d+\k$ is maximal abelian, then so is $\theta(\d)\subset \d+\k$, and so $\theta(\d)=\d$. Thus $\theta$  induces an automorphism of $\ell^\infty$. It is an easy exercise that these are given by permutations of $\mathbb{N}$; if $\sigma$ is a permutation of $\mathbb{N}$, then the corresponding automorphism is given by $\{x_n\}_{n\ge 1}\mapsto \{x_{\sigma(n)}\}_{n\ge 1}$. Clearly the operator $U_\sigma$ acting in $\h$, given by 
$$
U_\sigma\xi=\sum_{n\ge 1} \xi_n e_{\sigma(n)} \ , \ \ \hbox{ if } \xi=\sum_{n\ge 1} \xi_n e_n,
$$
is a unitary operator in $\h$. This unitary operator implements in turn the automorphism $\theta_\sigma=\theta$, 
$$
\theta_\sigma:\d+\k\to\d+\k \ , \ \ \theta_\sigma(T)=U_\sigma TU_\sigma^* \ , \hbox{ if }  T\in \d+\k.
$$ 
Note that $U_\sigma\in\d+\k$ if and only if the permutation $\sigma$ leaves all but a finite set of numbers fixed. Otherwise, $\theta_\sigma$ is an outer automorphism.

The following result  shows that all automorphisms of $\d+\k$ are {\it approximately inner} \cite{kadisonringrose}:
\begin{lem}
Let $\theta$ be a $*$-automorphism of $\d+\k$. Then there exists a unitary operator $U\in\b(\h)$ such that $\theta(T)=UTU^*$.
\end{lem}
\begin{proof}
By Lemma \ref{en K} and the subsequent remark, there exists a unitary operator $U\in\b(\h)$ such that $\theta(K)=UKU^*$ for all $K\in\k$. We claim that this unitary implements $\theta$ in the whole algebra $\d+\k$. Let $0\le D=D_\dd\in\d$, where $\dd=\{d_n\}_{n\ge 1}$. Consider $D_N$ given by the sequence $\dd_N$, which is the truncation of the sequence $\dd$ at $N$. Then it is clear that $D_N\le D$ and $D_N\to D$ in the strong operator topology. The first assertion implies that $\theta(D_N)\le \theta(D)$. On the other hand, since $D_N$ is compact, $\theta(D_N)=UD_NU^*$. Also it is clear that $UD_NU^*\to UDU^*$ in the strong operator toplogy. It follows that 
$$
UDU^*\le \theta(D).
$$
Pick now $D'_N=D_{\dd_N'}$, where $\dd_N'=\{(d'_N)_n\}_{n\ge 1}$ is given by 
$$
(d'_N)_n=\left\{ \begin{array}{l} d_n , \hbox{ if } n\le N \\ \|\dd\|_\infty , \hbox{ if } n > N \end{array} \right. .
$$
Clearly, $D'_N=D_N+\|\dd\|_\infty (I-E_N)$, where $E_N$ is the (diagonal) projection onto the subspace generated by $e_1,\dots , e_N$.
Also it is clear that $D'_N\ge D$. Therefore $\theta(D'_N)\ge \theta(D)$. Note that 
$$
\theta(D'_N)=\theta(D_N)+\|\dd\|_\infty (I-\theta(E_N))=UD_NU^*+\|\dd\|_\infty (I-UE_NU^*)
=U(D_N+\|\dd\|_\infty (I-E_N))U^*
$$
$$
=UD'_NU^*.
$$
A simple calculation shows that $D'_N\to D$ strongly, and  thus again $UD'_NU^*\to UDU^*$ strongly. Therefore
$UDU^*\ge \theta(D)$. Thus, $\theta(D)=UDU^*$ for all $D\in\d$ with $D\ge 0$. Then it holds for all $D\in\d$.
\end{proof}
\begin{coro}
All automorphisms of $\d+\k$ are strong (and weak) operator continuous. 
\end{coro}

Clearly there are unitaries in $\b(\h)$ that do not induce $*$-homomorphisms of $\d+\k$: pick a selfadjoint operator $A$ which does not belong to $\d+\k$; by the theorem of Weyl and von Neumann, there exists a unitary $W$ in $\b(\h)$ such that $W^*AW\in \d+\k$, then $W$ does not induce a $*$-homomorphism of $\d+\k$.
The condition on $U\in\u(\h)$ that $UDU^*\in\d+\k$ for all $D\in\d$  is clearly necessary for $\theta=Ad(U)$ to define an automorphism of $\d+\k$. Let us show that it is also sufficient:
\begin{prop}
Let $U\in\u(\h)$ such that $UDU^*\in\d+\k$, for all $D \in \d$. Then $Ad(U)$ defines an automorphism of $\d+\k$.
\end{prop}
\begin{proof}
Clearly $Ad(U)$ defines an injective $*$-homomorphism. Let us check that it is onto. Clearly $Ad(U)(\k)=\k$. So we must show that $\d\subset Ad(U)(\d+\k)$. Let $D\in\d$. Denote by $u=\pi(U)$, $d=\pi(D)$ in $\b(\h)/\k$. Clearly $Ad(u)$ is a $*$-automorphism of $\b(\h)/\k$, and the hypothesis implies that  $\b=Ad(u)(\pi(\d))$ is an  C$^*$-subalgebra of $\pi(\d)$. As remarked in Section \ref{preliminar}, $\pi(\d)\subset \d+\k$ is a maximal abelian subalgebra. Then $Ad(u)(\pi(\d))$, being the image of a masa is also a masa. Then $Ad(u)(\pi(\d))=\pi(\d)$. Thus, there exists $D'\in\d$ such that $Ad(u)(d')=d$, where $d'=\pi(D')$. Then there exists $K\in\k$ such that 
$$
D=UD'U^*+K=U(D'+U^*KU)U^*\in Ad(U)(\d+\k).
$$
\end{proof}
We know so far, explicitly, the existence of the following automorphisms:
\begin{enumerate}
\item
For $\ww=\{w_n\}_{n\ge 1}$, with $|w_n|=1$, $\theta_\ww=Ad(D_\ww)$.
\item
For $X^*=X\in\k$, $\theta_X=Ad(e^{iX})$.
\item
For $\sigma$ a permutation of $\mathbb{N}$, the automorphism $\theta_\sigma=Ad(U_\sigma)$ described above.
\end{enumerate}
Clearly, among these, only the $\theta_\sigma$, with $\sigma$ of infinite support  are outer. Note also that
$U_\sigma D_\ww =D_{\sigma(\ww)} U_\sigma$, where $\sigma(\ww)=\{w_{\sigma(n)}\}_{n\ge 1}$, and that $U_\sigma e^{iX}=e^{iU_\sigma X U_\sigma^*} U_\sigma$; so that
$$
\theta_\sigma \theta_\ww \theta_X=\theta_{\sigma(\ww)} \theta_{U_\sigma X U_\sigma^*} \theta_\sigma.
$$
Also, due to the factorization of $\u_{\d+\k}$ given in Proposition \ref{factores unitarios}, it is clear that any inner automorphism can be expressed as $\theta_\ww\theta_X$, for suitable $\ww$ and $X$. If $\sigma_1,\sigma_2$ are permutations of $\mathbb{N}$, $\theta_{\sigma_1}\theta_{\sigma_2}=\theta_{\sigma_1\sigma_2}$.
Thus, any autormorphism in the group  generated by these types of automorphisms, can be expressed as a product 
\begin{equation}\label{tetas}
\theta_\ww\theta_X\theta_\sigma.
\end{equation}
\begin{quest}\label{pregunta}
Does this  group exhaust the whole automorphism group of $\d+\k$. In other words, are all unitaries $U$ of $\b(\h)$ which satisfy $U\d U^*\subset\d+\k$ of the form $U=D_\ww e^{iX} U_\sigma$, with $X=X^*$ compact? 
\end{quest}
Let us denote by $\u_0(\h)$ the group generated by unitaries $U_\sigma$, $D_\ww$ and $e^{iX}$, for  $\sigma$  permutations of $\mathbb{N}$, $\ww$ of modulus $1$ and $X^*=X$ compact. Note that due to the above description, we can write
\begin{equation}\label{union}
\u_0(\h)=\bigcup_{\sigma\in\s(\mathbb{N})} U_\sigma \cdot\u_{\d+\k},
\end{equation}
where $\s(\mathbb{N})$ denotes the group of permutations of $\mathbb{N}$. Denote by $\s_f(\mathbb{N})\subset\s(\mathbb{N})$ the subgroup of permutation with finite support, i.e. which leave a co-finite set fixed.  The above description can be refined, namely:
\begin{equation}\label{union discreta}
\u_0(\h)=\bigcup_{[\sigma]\in\s(\mathbb{N}) / \s_f(\mathbb{N}) } U_\sigma \cdot\u_{\d+\k}.
\end{equation}
To prove this fact, we need the following Lemma:
\begin{lem}\label{lema eduardo}
Let $\sigma$ be a permutation of $\mathbb{N}$ of infinite support, and  $T=D+K\in\d+\k$. Then
$$
\|U_\sigma-T\|\ge 1.
$$
If additionally $T\in\u_{\d+\k}$, then 
$$
\|U_\sigma-T\|\ge \sqrt2.
$$
\end{lem}
\begin{proof}
Fix $\epsilon>0$. Let $E_N$ be the orthogonal projection onto the subspace generated by $e_1,\dots,e_N$. Then
$$
\|U_\sigma-T\|\ge\|E_N^\perp(U_\sigma-T)E_N^\perp\|=\|E_N^\perp(U_\sigma-D)E_N^\perp - E_N^\perp KE_N^\perp\|
$$
$$
\ge \|E_N^\perp(U_\sigma-D)E_N^\perp\|-\|E_N^\perp KE_N^\perp\|.
$$
Since $K$ is compact, we can pick $N$ large enough so that $\|E_N^\perp KE_N^\perp\|<\epsilon$. Then, for such $N$,
$$
\|U_\sigma-T\|\ge \|E_N^\perp(U_\sigma-D)E_N^\perp\|-\epsilon\ge \|E_N^\perp (U_\sigma-D)E_N^\perp e_m\|-\epsilon,
$$
for any  vector $e_m$ of the orthogonal basis. For the given $N$, we choose $m$ so that $m>N$ and  $m\ne\sigma(m)>N$. This is possible because $\sigma$ has infinite support. Then 
$$
E_N^\perp(U_\sigma-D)E_N^\perp e_m=e_{\sigma(m)}-d_me_m, 
$$
whose norm is $\sqrt{1+|d_m|^2}\ge 1$. Then $\|U_\sigma-T\|>1-\epsilon$, for any $\epsilon>0$. 

If $T$ is unitary in $\d+\k$, we can choose a decomposition $T=D+K$ with $|d_n|=1$ for all $n\in\mathbb{N}$. Then
$\|e_{\sigma(m)}-d_me_m\|=\sqrt2$. 
\end{proof}

\begin{teo}
\begin{equation}
\u_0(\h)=\bigcup_{[\sigma]\in\s(\mathbb{N}) / \s_f(\mathbb{N}) } U_\sigma \cdot\u_{\d+\k}.
\end{equation}
Thus, in the norm topology of $\b(\h)$,  $\u_0(\h)$ is a non countable discrete union of copies of $\u_{\d+\k}$. In particular, $\u_0(\h)$  is closed.
\end{teo}
\begin{proof}
Formula (\ref{union discreta}) states that in formula (\ref{union}), it suffices to choose one $\sigma$ in each   class of $\s(\mathbb{N}) / \s_f(\mathbb{N})$. This is clear, if $\sigma'=\sigma\gamma$, for some $\gamma\in\s_f(\mathbb{N})$, then $U_{\sigma'}=U_\sigma U_\gamma$. Clearly $U_\gamma\in\u_{\d+\k}$, and then $U_{\sigma'}\cdot \u_{\d+\k}=U_{\sigma}\cdot \u_{\d+\k}$. On the other hand, if $\sigma$ and $\sigma'$ belong to different classes, $U_{\sigma}^*U_{\sigma'}=U_{\sigma^{-1}\sigma'}\notin \u_{\d+\k}$, because $\sigma^{-1}\sigma'$ has infinite support.   Since $\s_f(\mathbb{N})$ is countable, the quotient $\s(\mathbb{N}) / \s_f(\mathbb{N})$ is uncountable. To end the proof, let us show that if $\sigma,\sigma'$ belong to different classes, then $\|U_\sigma V -U_{\sigma'} W\|\ge \sqrt2$, if $V, W\in \u_{\d+\k}$. Indeed
$$
\|U_\sigma V -U_{\sigma'} W\|=\|U_\sigma'(U_{\sigma'}^*U_\sigma-WV^*)V \|=\|U_{\sigma'^{-1}\sigma}-WV^*\|\ge \sqrt2,
$$
by Lemma \ref{lema eduardo}.
\end{proof}
Using Lemma \ref{lema eduardo}, we can also show the following:
\begin{prop}
Let $\sigma\in \s(\mathbb{N})$  of infinite support, and $U\in\u_{\d+\k}$. Then
$$
\|\theta_\sigma-Ad(U)\|\ge 2.
$$
\end{prop}
\begin{proof}
$$
\|\theta_\sigma-Ad(U)\|=\sup_{X\in\d+\k, \|X\|\le 1} \|U_\sigma XU_\sigma^*-UXU^*\|=\sup_{X\in\d+\k, \|X\|\le 1} \|U^*U_\sigma X-XU^*U_\sigma\|.
$$
Since $\d+\k$ is strongly dense in $\b(\h)$, it is clear that
$$
\sup_{X\in\d+\k, \|X\|\le 1} \|U^*U_\sigma X-XU^*U_\sigma\|=\sup_{X\in\b(\h), \|X\|\le 1} \|U^*U_\sigma X-XU^*U_\sigma\|.
$$
This, in turn, is the norm of the derivation $\delta_{U^*U_\sigma}$,   where $\delta_A(X)=XA-AX$. By a result by J.G. Stampfli \cite{stampfli}, $\|\delta_A\|=2 \inf_{\lambda\in\mathbb{C}} \|A-\lambda I\|$. Then
$$
\|\theta_\sigma-Ad(U)\|=2 \inf_{\lambda\in\mathbb{C}} \|U^*U_\sigma-\lambda I\|=2 \inf_{\lambda\in\mathbb{C}} \|U_\sigma-\lambda U\|.
$$
By Lemma \ref{lema eduardo}, $\|U_\sigma-\lambda U\|\ge 1$, and the proof follows.
\end{proof}
\begin{rem}
Therefore, as in  formula (\ref{union discreta}), we can describe the group of automorphisms generated by unitaries in $\u_0(\h)$ as a discrete union of copies of the group of inner automorphisms,
$$
\{Ad(U): U\in\u_0(\h)\}=\bigcup_{[\sigma]\in\s(\mathbb{N}) / \s_f(\mathbb{N}) } Ad(U_\sigma) \cdot\{Ad(V): V\in\u_{\d+\k}\}.
$$ 
If $\theta_1,\theta_2$ belong to different copies, then $\|\theta_1-\theta_2\|\ge 2$. Indeed, $\theta_1=Ad(U_{\sigma_1} U_1)$ and $\theta_2=AD(U_{\sigma_2}U_2)$, with $\sigma_2^{-1}\sigma_1$ of infinite support. Then
$$
\|Ad(U_{\sigma_1}U_1)-Ad(U_{\sigma_2}U_2)\|=\|Ad(U_{\sigma_2})\left(Ad(U_{\sigma_2^{-1}\sigma_1}-Ad(U_2U_1^*)\right)Ad(U_1)\|
$$
$$
=\|Ad(U_{\sigma_2^{-1}\sigma_1})-Ad(U_2U_1^*)\|\ge 2,
$$
by the above proposition. In particular, the group $\{Ad(U): U\in\u_0(\h)\}$ is open and closed in the group of all $*$-automorphisms of $\d+\k$, in the norm topology.
\end{rem}
\begin{rem}
The fact that the automorphisms of $\d+\k$ leave $\k$ invariant, implies that any automorphism $\theta$ of $\d+\k$ induces an automorphism $\bar{\theta}$ of the quotient $\ell^\infty/\ccc$. The automorphisms of $\ell^\infty/\ccc\simeq C(\beta\mathbb{N}\setminus \mathbb{N})$ are in one to one correspondence, by Gelfand's map, with the homeomorphisms of the residual set $\beta\mathbb{N}\setminus\mathbb{N}$. It is known, if one assumes the continuous hypothesis, that the set of homeomorphisms of  $\beta\mathbb{N}\setminus\mathbb{N}$ has cardinality $2^c$ (see for instance \cite{weaver}, Chapter 15)

On the other hand, since $\h$ is separable, $\b(\h)$, and thus $\u(\h)$, has cardinality $c$. Since all automorphisms of $\d+\k$ are implemented by unitaries, it follows that they have cardinality $c$. 

Therefore not every automorphism of $\d+\k /\k$ can be lifted to an automorphism of $\d+\k$. Clearly the former set is clearly much more complicated.
\end{rem}

\begin{rem}
There are plenty of  (non onto) $*$-endomorphisms. To the characters (onto the subalgebra $\mathbb{C}\cdot 1\subset \d+\k$), one can add the following. Let $F\subset\mathbb{N}$, and consider $\d_F\subset\d$ the subalgebra
$$
\d_F=\{D_\xx: x_i=x_j \hbox{ if } i,j\in F\}.
$$
$\d_F$ is the unitization of the ideal of $\d$ of diagonal matrices whose $F$-entries are zero. Recall the multiplicative functionals $\Psi_\m=\psi_\m\circ\pi:\k+\d\to \mathbb{C}$, where $\psi_\m$ is the multiplicative functional associated to the maximal ideal $\m$. Fix $\m_0$ and let $\{\m_j: j\in\mathbb{N}\setminus F\}$, where $\m_j\ne\m_0$ and $\m_j\ne\m_k$ if $k\ne j$, but otherwise arbitrary. Then the map
$$
\varphi:\d+\k\to\d\subset \d+\k \ , \ \ \varphi(T)=D_\dd \  , \ \hbox{ where } d_n=\left\{ \begin{array}{l} \Psi_{\m_0}(T) \hbox{ if } n\in F \\ \Psi_{\m_n} \hbox{ if } n\notin F \end{array} \right.  .
$$   
is clearly a $*$-endomorphism of $\d+\k$, whose image if $\d_F$. Note also that 
$$
\k\subset N(\varphi)=\cap_{n\in F} \m_n \cap \m_0.
$$
\end{rem}
\section{Projections}
Finite rank projections in $\b(\h)$ belong to $\d+\k$, as well as arbitrary projections in $\d$ (diagonal matrices with entries $0$ or $1$). 
\begin{lem}
Let $P=D+K\in\d+\k$ be an orthogonal projection, with $D^*=D=D_\dd$. Then the only possible accumulation points of  the sequence $\dd=\{d_n\}_{n\ge 1}$ are $0$ or $1$. 
\end{lem}
\begin{proof}
Since $D+K=D^2+DK+KD+ K^2$, it follows that $D-D^2$ is a compact self-adjoint 
operator. Thus $\sigma(D-D^2)=\{d_n-d_n^2: n\ge 1\}$ is a sequence whose only 
possible accumulation point is  $0$.
\end{proof}
As a consequence, we have,
\begin{prop}
Let $P\in\d+\k$ be an orthogonal projection. Then there exists a decomposition  $P=E+K$, where $E\in\d$ is a (diagonal) projection.
\end{prop}
\begin{proof}
Let $P=D_\dd+K'$, with $\dd=\{d_n\}_{n\ge 1}$. Then we can write  $\dd$ as the union of  two subsequences. Consider $\{d_{j_k}\}$ and $\{d_{i_k}\}$ with $d_{j_k}\to 0$, $d_{i_k}\to 1$ and $\mathbb{N}=\{j_k: k\ge 1\}\cup\{i_k: k\ge 1\}$ (if $0$ or $1$ do not occur as accumulation points, we omit the corresponding subsequence).  Let $\dd_0=\{d'_n\}_{n\ge 1}$ and $\dd_1=\{d''_n\}_{n\ge 1}$ given by
$$
d'_n=\left\{ \begin{array}{ll} d_{j_k} & \hbox{ if } n=j_k \\ 0 & \hbox{ if not} \end{array} \right.  \ , \ \ d''_n=\left\{ \begin{array}{ll} d_{i_k}-1 & \hbox{ if } n=i_k \\ 0 & \hbox{ if not} \end{array} \right.
$$
Clearly $D_{\dd_0}$ and $D_{\dd_1}$ are compact. Let $E=D_\tt$, where $\tt=\{t_n\}_{n\ge 1}$ is given by
$$
t_n=\left\{ \begin{array}{ll} 0 & \hbox{ if } n=j_k \\ 1 & \hbox{ if } n=i_k \end{array} \right. .
$$
Then $E$ is a projection in $\d$, and clearly
$$
P=E+K'+D_{\dd_0}+D_{\dd_1}.
$$
\end{proof}
The set $\pdk$ of  projections in $\d+\k$ consists of  three disjoint classes:
$$
\p^0_{\d+\k}=\{P \hbox{ has finite rank}\} \ , \ \ \p^1_{\d+\k}=\p_1=\{P \hbox{ has co-finite rank}\},
$$
and the complement of the union of these sets, the set $\p^\infty_{\d+\k}$. Note that the first two classes correspond with projections $P=D+K$ such that the spectrum of $D$ accumulates only at (respectively) $0$ or $1$, whereas in the class $\p^\infty_{\d+\k}$  the spectrum of $d$ accumulates both at $0$ and $1$. We shall focus on the description of this latter class.

It will be useful to recall the definition of the {\it restricted Grassmannian} \cite{sato},\cite{segalwilson}
(also called Sato Grassmannian)

Let $\h=\h_+\oplus \h_-$ an orthogonal decomposition of $\h$, with $\dim \h_+=\dim \h_-=\infty$. Denote by $E_+,E_-$ the orthogonal projections onto $\h_+,\h_-$, respectively. A projection $P$ belongs to the restricted Grassmannian $Gr_{res}(\h_+)$ with respect to the  subspace $\h_+$  if and only if
\begin{enumerate}
\item
$$
E_+P|_{R(P)}:R(P)\to \h_+\in\b(R(P),\h_+)
$$
is a Fredholm operator in $\b(R(P),\h_+)$, and
\item
$$
E_-P|_{R(P)}:R(P)\to \h_-\in\b(R(P),\h_-)
$$
is compact.
\end{enumerate}
The index of the first operator characterizes the connected components of $Gr_{res}(\h_+)$: two projections $P,Q$ belong to the same component of $Gr_{res}(\h_+)$ if and only if  $E_+P\in\b(R(P),\h_+)$ and $E_+Q\in\b(R(Q),\h_+)$ have the same Fredholm index.

\begin{rem}
Given $P\in Gr_{res}(\h_+)$, the Fredholm  index of  $E_+P\in\b(R(P),\h_+)$ was called  the index of the pair $(P,E_+)$ in \cite{ass}. There, among other properties, it was shown that it can be computed as
$$
index (P,E_+)=\dim(R(P)\cap N(E_+)) - \dim(N(P)\cap R(E_+)).
$$
\end{rem}
\begin{prop}
Let $P\in\d+\k$ be a  projection in the class $\p^\infty_{\d+\k}$, $P=E+K$ with $E$ a projection. Then $P\in Gr_{res}(R(E))$.
\end{prop}
\begin{proof}
The null space of $EP|_{R(P)}$ is $R(P)\cap N(E)$, which, after elementary 
computations,  coincides with $N(P-E-I)=N(K-I)$, which is finite dimensional. 
The orthogonal complement of the range  of $EP|_{R(P)}$ is $N(P)\cap R(E)$, 
which coincides with $N(P-E+I)=N(K+I)$, which is also finite dimensional.
\end{proof}
\begin{rem}
Note  that in the above proposition, $index(P,E)=\dim N(K-I)-\dim N(K+I)$. The compact operator $K=P-E$, being a difference of projections, satisfies that $\sigma(K)\setminus \{-1,1\}$ is symmetric with respect to the origin, with $\lambda$ having the same (finite) multiplicity as $-\lambda$ ($|\lambda|<1$). However, there is no restriction for the multiplicities of $\pm1$, other than finiteness. Thus one can find examples where any value of the index can occur.
\end{rem}
It is clear that if $P\in\p^\infty_{\d+\k}$, there are infinitely many ways to 
decompose $P=E+K$: one can subtract from $E$ any finite number of $1$´s, and 
add the corresponding finite rank projection to $K$. Also it is clear, for 
instance playing with both $P$ and $E$ diagonal, that the index is not 
conserved for different decompositions of the same projection. 

Conversely, if $P=E+K=F+K'$ with $E,F$ diagonal projections, then $E-F$ is compact, which means that the subspace where they differ, i.e. $ R(E)\cap N(F)\oplus N(E)\cap R(F)$, must have finite dimension (indeed, in this subspace $E-F$ is a compact  symmetry). 

Nevertheless the index does in fact play a key role in determining the connected components of $\p^\infty_{\d+\k}$.
Let us recall the following property from \cite{ass}, which we state as a lemma, restricting the hypothesis to our current problem: namely for pairs of projections with compact difference. Recall that such pairs of projections have finite index.
\begin{lem} {\rm (\cite{ass})} 
Let $(P,Q)$ and $(Q,R)$ be two pairs of projections  with $P-Q, Q-R\in\k$. Then $P-R\in\k$ and
$$
index(P,R)=index(P,Q)+index(Q,R).
$$
\end{lem}
\begin{teo}\label{teo 1}
Let $P\in\pdk$. Then there exists a diagonal projection $E_0$ such that $P-E_0\in\k$ and $index(P,E_0)=0$. Moreover, 
there exists a unitary  operator $U$ such that $U-I\in\k$ (i.e., $U=e^{iX}$ with $\|X\|\le\pi$, $X\in\k$; in particular, $U\in\u_{\d+\k}$), which satisfies
$$
UE_0U^*=P.
$$
\end{teo}
\begin{proof}
Suppose first that $P\in\p^\infty_{\d+\k}$. Let $E$ be a diagonal projection such that $P-E$ is compact. Let $m=index(P,E)$. Clearly, there exists a diagonal  projection $E_0$ such that $E-E_0$ has finite rank and $ind(E,E_0)=-m$. Then, by the above lemma,
$$
index(P,E_0)=index(P,E)+index(E,E_0)=0.
$$
Denote by $K_0=P-E_0\in\k$. 
Consider the decomposition of $\h=\h_0\oplus\h_1$, where $\h_0=N(K_0^2-1)^\perp$ and $\h_1=N(K_0^2-1)$, which reduces both $P$ and $E_0$. In $\h_0$, $\pm 1\notin \sigma(K_0|_{\h_0})$, thus, $\|P|_{\h_0}-E|_{\h_0}\|<1$. 
It is well known that two projections in a C$^*$-algebra at distance less than one are conjugate by a unitary operator in the algebra (even more \cite{pr}: projections at distance less than one are conjugated by the exponential map of the manifold of projections).  In \cite{sato} it was shown that in the case of the restricted Grassmannian,  the argument at the exponential is compact. Namely, there exists $X_0^*=X_0\in\k(\h_0)$ with $ \|X_0\|<\pi/2$ such that $P|_{\h_0}=e^{iX_0}E_0|_{\h_0}e^{-iX_0}$.

In $\h_1$, $P|_{\h_1}=1\oplus 0$ and $E_0|_{\h_1}=0\oplus 1$, in the decomposition $\h_1=N(K_0-I)\oplus N(K_0+I)$. Note that  $index(P,E_0)=0$ means that both  subspaces have the same  (finite) dimension. Then it is clear that there exists $X_1^*=X_1$, with norm $\|X_1\|=\pi/2$, acting in $\h_1$, such that $e^{iX_1}E_0|_{\h_1}e^{-iX_1}=P|_{\h_1}$. For instance, pick any unitary isomorphism $V:N(K+I)\to N(K-I)$, and put $X_1=i\frac{\pi}{2}(V-V^*)$.  Then 
$$
X=X_0\oplus X_1
$$
is  compact, self-adjoint, $\|X\|\le \pi/2$, and satisfies 
$e^{iX}E_0e^{-iX}=P$. The proof in the case  $P\in\p^0_{\d+\k}$ is 
straightforward. Note that $P\in\p^1_{\d+\k}$ if and only if 
$1-P\in\p^0_{\d+\k}$, which deals with the remaining case.
\end{proof}
\begin{rem}
The projection $E_0$  of the above result is clearly non unique: given $E$ there are infinitely many diagonal projections $F$ with $E-F$ of finite rank and
$$
-m=index(E,F)=\#\{n\in\mathbb{N}: E_{n,n}=1 \hbox{ and } F_{n,n}=0\} - \# \{n\in\mathbb{N}: E_{n,n}=0 \hbox{ and } F_{n,n}=1\}.
$$
Then $P$ is conjugated with many different diagonal projections by means of the exponential map. The different diagonal projections $E_0, F_0$ with $index(P,E_0)=index(P,F_0)=0$, satisfy (by means of the same lemma above)
$$
index(E_0,F_0)=0.
$$
Thus, again using the facts cited from \cite{sato}, they are also conjugate via the exponential map of $\pdk$: there exists $Y=Y^*\in\k$ with $\|Y\|\le\pi$ such that $e^{iY}E_0e^{-iY}=F_0$. 
\end{rem}
The following result will be useful to characterize the connected components of $\p^\infty_{\d+\k}$.
\begin{lem}
Let $E$, $F$ be diagonal projections such that $E-F$ is non compact, or $E-F$ is compact but $index(E,F)\ne 0$. Then $E$ and $F$ are not conjugate by unitaries in $\u_{\d+\k}$. 
\end{lem} 
\begin{proof}
As seen above, unitaries $U$ in $\d+\k$ can be written $U=De^{iX}$, with $D\in\d$ unitary and $X=X^*$ compact. Therefore, if $E$ and $F$ are conjugate in $\d+\k$, then $F=UEU^*=e^{iX}Ee^{-iX}$. Since $e^{iX}=I+K$ for some $K\in\k$, it follows that 
$$
E-F=E-(I+K)E(I+K^*)=-KE-EK^*-KEK^*\in\k,
$$
which rules out the case $E-F$ non compact. Consider the path $E(t)=e^{itX}Ee^{-itX}$. Clearly, by the same argument above, $E-E(t)\in\k$ for all $t$. Thus $index(E,E(t))$ is defined. We claim that it is constant. Recall that $index(E,E(t))$ can be computed  as the Fredholm index of the operator
$$
E(t)E|_{R(E)}:R(E)\to R(E(t)).
$$
Since $E(t)=e^{itX}Ee^{-itX}$ and $R(E(t))=e^{itX}(R(E))$, the index of this operator coincides with the index of 
$$
Ee^{-itX}E|_{R(E)}:R(E)\to R(E).
$$
Since the index of a Fredholm operator (in a fixed Banach space) is locally constant, \\ $index(E,E(t))$ is constant. In particular, at $t=0$ and $t=1$,
$$
0=index(E,E)=index(E,F).
$$
\end{proof}
\begin{coro}
The connected components of $\p^\infty_{\d+\k}$ coincide with the orbits of the exponential  subgroup 
$$
\u_\infty(\h)=\{e^{iX}: X=X^*\in\k\}.
$$
They are parametrized by the set of diagonal projections (with infinite $0$'s and infinite $1$'s), modulo the index: $E$ and $E'$ define the same component if and only if $index(E,E')=0$. In particular, $\pdk$ has uncountable many connected components.
\end{coro}
\begin{proof}
Only the last assertion needs a proof. Clearly, there are uncountable many diagonal projections with infinite $0$'s and infinite $1$'s. The set of diagonal projections $E'$ which have zero index with respect to a fixed projection $E_0$ is countable. Indeed, for $k\in\mathbb{N}\cup \{0\}$ let 
$$
\e_k=\{E': index(E',E_0)=0 \hbox{ and } \dim(R(E')\cap N(E_0))=k\}.
$$
Clearly the set $\e_k$ is countable: it has the same  cardinality as the set of the subsets of $A\subset\mathbb{N}$ with $\# A=k$, which is countable.  Also it is clear that
$$
\{E': index(E',E_0)=0\}=\cup_{k=0}^\infty \e_k.
$$
\end{proof}

Let us make a brief recollection of the basic facts of the metric geometry of the space of projections in $\b(\h)$. Our references for these facts are \cite{pr} and \cite{pemenoscu}.
\begin{rem}
The distance between two projections $P,Q\in\b(\h)$, can be computed as
$$
d(P,Q)=\inf\{\ell(\gamma): \gamma \hbox{ is continuous and piecewise differentiable, and joins } P \hbox{ and } Q\},
$$
where it is implicit that $\gamma(t)$ consists of projections. If $\gamma$ is parametrized in the interval $I$, the length $\ell(\gamma)$ is computed 
$$
\ell(\gamma)=\int_I \|\dot{\gamma}(t)\| dt.
$$
Pairs of projections which can be joined by a minimal curve have been 
characterized: there exists a curve of minimal length if and only if  
$index(P,Q)=0$. In that case, there exists a self-adjoint operator $X$ of norm 
$\|X\|\le \pi/2$, satisfying that $X(R(P))\subset N(P)$ (and therefore also 
$X(N(P))\subset R(P)$),  with  the same condition for $Q$; we shall abbreviate 
these conditions by saying that $X$ is $P$ and $Q$ co-diagonal. This operator 
$X$ satisfies that 
$$
e^{iX}Pe^{-iX}=Q.
$$
The minimal curve $\delta$ joining $P$ and $Q$ is $\delta(t)=e^{itX}Pe^{-itX}$, its length is
$$
\ell(\delta)=d(P,Q)=\|X\|=\sin^{-1}(\|P-Q\|).
$$
\end{rem}
These facts and the proof of Theorem \ref{teo 1} imply the following:
\begin{coro}
Let $P,Q$ be projections in $\d+\k$, in the same connected component. Then there exists  $X^*=X$ with the following properties
\begin{enumerate}
\item
$X$ is  compact, with $\|X\|\le \pi/2$.
\item
The curve $\delta(t)=e^{itX}Pe^{-itX}$, joins  $\delta(0)=P$ and $\delta(1)=Q$ in $\p_{\d+\k}$.
\item
$\delta$ has minimal length among all piecewise differentiable curves of projections of $\b(\h)$ which join $P$ and $Q$. In particular, it has minimal length among curves of projections in $\d+\k$.
\end{enumerate}
\end{coro}
\begin{proof}
We consider first the case when $P,Q\in\p^\infty_{\d+\k}$. 
There exist diagonal projections $E_0$, $E_1$ such that $index(P,E_0)=index(Q,E_1)=0$. Since $E_0$ and $E_1$ lie in the same connected component, it also holds that $index(E_0,E_1)=0$. Then $index(P,Q)=0$. Also since the differences $P-E_0, Q-E_1$ and $E_0-E_1$ are compact, also the difference $K=P-Q$ is compact. Using the same reasoning as in the proof of Theorem  \ref{teo 1}, namely decomposing $\h$ as
$$
\h=N(K^2-I)^\perp \oplus N(K^2-I),
$$
we have that in $N(K^2-I)^\perp$ the reductions of $P$ and $Q$ lie at distance 
strictly less than $1$, and thus one can find a  compact self-adjoint operator 
$X_0$, which is co-diagonal with respect to these reduced projections, 
satisfies $\|X_0\|<\pi/2$, and $e^{iX_0}$ conjugates the (reduced) projections. 
As we saw, the operator $X_1$ performing the analogous task in the finite 
dimensional space $N(K^2-I)$ can also be chosen co-diagonal (see 
\cite{pemenoscu} for an explicit description of $X_1$). In this case $X_1$ (if 
non trivial), has norm equal to $\pi/2$. Then $X=X_0\oplus X_1$ is the 
self-adjoint exponent of the minimal curve.

The cases $\p^0_{\d+\k}$ and $\p^1_{\d+\k}$ are straightforward.
\end{proof}

Projections $P\in\b(\h)$ with the same nullity $n(P)$ (dimension of the nullspace) and rank $r(P)$ (dimension of the range) are unitarily equivalent in $\b(\h)$. The next result establishes that projections in $\d+\k$ with the same nullity and rank are conjugated by an automorphism of $\d+\k$:
\begin{prop}
Let $P,Q\in\p_{\d+\k}$ such that $n(P)=n(Q)$ and $r(P)=r(Q)$, Then there exist an automorphism $\theta$ of $\d+\k$ such that $\theta(P)=\theta(Q)$.
\end{prop}
\begin{proof}
By the above results, there exist diagonal projections $E=D_\ee, F=D_\ff$ and 
self-adjoint compact operators $X,Y$ such that 
$$
P=e^{iX}Ee^{-iX} \ \hbox{ and } \ Q=e^{iY}Fe^{-iY}.
$$
Thus $E$ and $F$ satisfy $n(E)=n(F)$ and $r(E)=r(F)$. In other words, 
$$
\#\{n\in\mathbb{N}: e_n=0\}=\#\{n\in\mathbb{N}: f_n=0\} \hbox{ and } \#\{n\in\mathbb{N}: e_n=1\}=\#\{n\in\mathbb{N}: f_n=1\}.
$$
Then there exists a permutation $\sigma$ of $\mathbb{N}$ such that 
$$
\sigma(\{n\in\mathbb{N}: e_n=0\})=\{n\in\mathbb{N}: f_n=0\} \hbox{ and } \sigma(\{n\in\mathbb{N}: e_n=1\})=\{n\in\mathbb{N}: f_n=1\}.
$$
That is, 
$$
\theta_\sigma(E)=U_\sigma E U_\sigma^*=U_\sigma E U_{\sigma^{-1}}=F.
$$
Thus, $Q=e^{iY}\theta_\sigma(E)e^{-iY}=e^{iY}\theta_\sigma(e^{-iX}Pe^{iX})e^{-iY}=\theta_Y\theta_\sigma\theta_{-X}(P)$.
\end{proof}

Next we address a problem which is related to the question \ref{pregunta} of the previous section. Let  $U$ a unitary operator which implements an automorphism of $\d+\k$. This means that for any diagonal operator $D$  there exist a diagonal $D'$ and a compact $K$ such that $UDU^*=D'+K$.
\begin{lem}
Let $D_0$ be a diagonal operator with finite spectrum, and $U$ a unitary 
operator implementing an automorphism of $\d+\k$. Then there exist a compact 
self-adjoint operator $X_0$ and a permutation $\sigma_0$ of $\mathbb{N}$ such 
that $U_0=e^{iX_0}U_{\sigma_0}$ verifies
$$
UD_0U^*=U_0D_0U_0^*.
$$
\end{lem}
\begin{proof}
Let $D_0=\sum_{k=1}^{n}d_k E_k$, with $E_k$ mutually orthogonal diagonal 
projections. Consider the projection $UE_1U^*$ in $\d+\k$. By the same argument 
as in the previous proposition,  there exist a compact self-adjoint operator 
$X_1$ and a permutation $\sigma_1$ such that 
$$
UE_1U^*=e^{iX_1}U_{\sigma_1}E_1(e^{iX_1}U_{\sigma_1})^*.
$$
Denote by $U_1=(e^{iX_1}U_{\sigma_1})^*U$. Then $U_1E_1U_1^*=E_1$. Then $U_1$ 
is reduced by $R(E_1)$ , which is spanned by a subset of the fixed orthonormal 
basis. Consider $D_2=\sum_{k=2}^n d_k E_k$ and $U_1D_2U_1^*\in\d+\k$. By the 
same argument, fixing now $E_2$, there exist a compact self-adjoint operator 
$X_2$ of $R(E_1)^\perp$,  and a permutation $\sigma_2$ which leaves 
$\{k\in\mathbb{N}: e_k\in R(E_1)\}$ fixed, such that
$$
U_1E_2U_1^*=e^{iX_2}U_{\sigma_2}E_2(e^{iX_2}U_{\sigma_2})^*.
$$
Denote $U_2=(e^{iX_2}U_{\sigma_2})^*U_1$. Then 
$$
U_2E_1U_2^*=E_1,
$$
because  $e^{iX_2}U_{\sigma_2}$ acts as the identity in $R(E_1)$, and 
$$
U_2E_2U_2^*=E_2
$$
by construction. Iterating (finitely many times) this procedure, we obtain a unitary 
$$
U_0=e^{iX_n}U_{\sigma_n}e^{iX_{n-1}}U_{\sigma_{n-1}}\dots e^{iX_1}U_{\sigma_1},
$$
which as in (\ref{tetas}), can be rewritten in the form
$$
U_0=e^{iX_0}U_{\sigma_0},
$$
for $X_0$ self-adjoint and compact and $\sigma_0$ a permutation of 
$\mathbb{N}$, such that 
$$
U_0^*U E_k(U_0^*U)^*=E_k,
$$
for all $k=1,\dots, n$, and thus $U_0D_0U_0^*=UD_0U^*$.
\end{proof}
Clearly, the unitary operator $U_0$ above depends on $D$. 
\begin{coro}
Let $U$ be a unitary operator implementing an automorphism of $\d+\k$, and $D$ be an arbitrary diagonal operator.  Then there exist unitary operators $U_n=e^{iX_n}U_{\sigma_n}$, with $X_n=X_n^*$ compact,  such that 
$$
U_nDU_n^*\to UDU^*
$$
in norm.
\end{coro} 
\begin{proof}
Recall from Remark \ref{remark preliminar}, that diagonal operators with finite spectrum are dense in $\d$. Let $D_n$ be a sequence of diagonal operators with finite spectrum such that $D_n\to D$ in norm. From the above Lemma, we know that there exist  $U_n=e^{iX_n}U_{\sigma_n}$ such that 
$$
U_nD_nU_n^*=UD_nU^*.
$$
Then $U_nD_nU_n^*\to UDU^*$. On the other hand, 
$$
\|U_nDU_n^*-U_nD_nU_n^*\|=\|U_n(D-D_n)U_n^*\|=\|D-D_n\|\to 0.
$$
\end{proof}

\begin{rem}
The above result could be phrased as saying that automorphisms are pointwise inner  on diagonal elements with finite spectrum, and asymptotically given by conjugation with unitaries of the form $U_0=e^{iX_0}U_{\sigma_0}$ (with $X_0$ compact). This  clearly falls short of answering Question \ref{pregunta}. However, suppose that  for any diagonal $D$, there exists a unitary $U_0$ as above such that 
$UDU^*=U_0DU_0^*$. Then, choosing a unitary $D=D_\dd$ with $d_j\ne d_k$ when $k\ne j$,  this would imply that
$U_0^*U$ is  a diagonal unitary operator. This would answer our question in the affirmative.
\end{rem}

\section{Topology of the unitary group}

Recall that $\u_\infty(\h)=\{U\in\u(\h): U-I\in\k\}=\{e^{iX}:X^*=X\in\k\}$, and that  $\u_\d$ denotes the unitary group of $\d$. Clearly $\u_\infty(\h)$ and $\u_\d$ are closed subgroups of $\u_{\d+\k}$. 
Consider the following map
\begin{equation}\label{fibrado}
\pp:\u_\d\times \u_\infty(\h) \to \u_{\d+\k} \ , \ \ \pp(D,V)=DV.
\end{equation}
As seen in Proposition \ref{factores unitarios}, $\pp$ is onto. We shall see 
that it is a submersion and a fibre bundle. Clearly it is a $C^\infty$ map. In 
order to prove that it is a submersion, let us show that it has $C^\infty$ 
local cross sections. Since  the spaces concerned are groups, clearly, it 
suffices to show the existence of a local cross section on a neighborhood of 
the identity:
\begin{lem} There is a map $\ss:\b_2=\{U\in\u_{\d+\k}: \|U-I\|<2\}\to \u_\infty(\h)\times \u_\d$ such that
$$
\pp\circ \ss=id_{\b_2}.
$$
\end{lem}
\begin{proof}
First note, by general considerations on C$^*$-algebras,  that $\|U-I\|<2$, implies that there exists a unique $Z^*=Z\in\d+\k$ with $\|Z\|<\pi$ such that $e^{iZ}=U$, and the (logarithm) map $\b_2\ni U\mapsto Z$ is $C^\infty$.
Note that 
$$
U=e^{iZ}=e^{i\Delta(Z)+i(Z-\Delta(Z))}=e^{i\Delta(Z)}+K,
$$
for  $K=e^{iZ}-e^{i\Delta(Z)}\in\k$, because $Z-\Delta(Z)\in\k$. Then 
$$
e^{iZ}=e^{i\Delta(Z)}(I+e^{-i\Delta(Z)}K),
$$
with $e^{i\Delta(Z)}\in\u_\d$ and $(I+e^{-i\Delta(Z)}K)\in\u_\infty(\h)$. That is, 
$$
\ss(U)=\left(e^{i\Delta(Z)}, I+e^{-i\Delta(Z)}(e^{iZ}-e^{i\Delta(Z)})\right)=(e^{i\Delta(Z)}, e^{-i\Delta(Z)}e^{iZ})
$$
is a $C^\infty$ cross section for $\pp$.
\end{proof}

\begin{rem}
The fibre of $\pp$ over $I$ is the set
$$
\f=\pp^{-1}(\{I\})=\{(D,D^*)\in\u_\d^2: D\in\u_\infty(\h)\}.
$$
Clearly $\f$ is an abelian group, and the unitaries $D$ are of the form $D_\zz$, with $\zz=\{z_n\}$ such that $|z_n|=1$ and $z_n\to 1$. Therefore 
$$
\f\simeq\{\zz=\{z_n\}: |z_n|=1, z_n\to 1\}=\u_{C(\alpha\mathbb{N})},
$$
the unitary group  of  the  continuous functions on the one point  compactification $\alpha\mathbb{N}=\mathbb{N}\cup\{\infty\}$ of $\mathbb{N}$.
\end{rem}
Since the fibres of $\pp$ are groups, and $\pp$ has smooth local cross 
sections, it follows that
\begin{coro}
The map $\pp$ is a group-quotient map with local cross sections, thus a 
submersion and a locally trivial fibre bundle.
\end{coro} 
Clearly, the group $\f$ is connected: a sequence $\zz=\{z_n\}$ in $\mathbb{T}$ 
which accumulates only at $1$ can be lifted to a sequence of real numbers 
accumulating at $0$, $z_n=e^{is_n}$ with $s_n\to 0$. Then the curve 
$\zz(t)=e^{its_n}$ connects $\zz$ with the identity element of this group. This 
implies, by means of the homotopy exact sequence of the fibre bundle $\pp$, 
that $\pi_1(\u_{\d+\k})$ is the image of the $\pi_1$ group of 
$\u_\d\times\u_\infty(\h)$:
$$
\pi_1(\f)\to \pi_1(\u_\d)\times\pi_1(\u_\infty(\h))\stackrel{\pi_1(\pp)}\longrightarrow\pi_1(\u_{\d+\k})\to 0.
$$
The homotopy groups of $\u_\infty(\h)$ have been computed, see for instance \cite{palais}. We shall be concerned with the first two homotopy groups.  As remarked, $\pi(\u_\infty(\h))=0$, and $\pi_1(\u_\infty(\h))=\mathbb{Z}$.
\begin{rem}
Let us characterize the first homotopy groups of $\f$ and $\u_\d$. This can be done by means of the exponential map, which in both cases gives the universal covering of these groups:
\begin{enumerate}
\item
For this purpose, we identify $\f$ with $\u_{C(\alpha\mathbb{N})}$. Then 
$$
\exp_\f:\cc_0^\mathbb{R}\to\f \ , \ \ \exp_\f(\xx)=D_{e^{2\pi i\xx}},
$$
where $\cc_0^\mathbb{R}$ denotes the space of real sequences which tend to zero, and $e^{2\pi i\xx}$ is the sequence $\{e^{2\pi ix_n}\}$. This map is clearly a covering map, the fibre over $I$ is the set of sequences of integers which tend to zero, i.e. the (additive) group $\cc_{00}^\mathbb{Z}$ of  sequences of integers which are zero but for  finite terms. Thus
$$
\pi_1(\f)\simeq\exp_\f^{-1}(I)\simeq \cc_{00}^\mathbb{Z}
$$
\item
The same map, defined over $\ell_\mathbb{R}^\infty$, the space of bounded real sequences, gives the universal covering of $\u_\d$,
$$
\exp_\d:\ell_\mathbb{R}^\infty\to \u_\d \ , \ \  \exp_\d(\xx)=D_{e^{2\pi i \xx}}.
$$
The fibre over $I$ in this case is the set of sequences of integers which are 
bounded (and thus assume a finite set of values), let us denote this group by 
$\ell^\infty(\mathbb{N},\mathbb{Z})$. Thus 
$$
\pi_1(\u_\d)\simeq \ell^\infty(\mathbb{N},\mathbb{Z}).
$$
\end{enumerate}
\end{rem}
Therefore, in order to compute the first homotopy group of $\u_{\d+\k}$, we must identify the image of the inclusion map
$$
\iota:\f\to \u_\d\times \u_\infty(\h) \ , \ \ \iota(\xx)=(D_{e^{2\pi i \xx}},D_{e^{-2\pi i \xx}})
$$ 
at the $\pi_1$-level,
$$
\pi_1(\iota): \cc_{00}^\mathbb{Z} \to \ell^\infty(\mathbb{N},\mathbb{Z}) \times \mathbb{Z}.
$$ 
\begin{prop}
With the above notations and identifications,
$$
\pi_1(\iota)(k_1,\dots,k_n,0,\dots)=\left((k_1,\dots, k_n,0, \dots), -\sum_{j=1}^n k_j\right).
$$
In particular, $\pi_1(\iota)$ is injective. 
\end{prop}
\begin{proof}
The sequences of the form $\ee_j=(0,\dots,1, 0, \dots)$ form a set of generators for the group $\cc_{00}^\mathbb{Z}$. They correspond to loops of the form
$$
\epsilon_j(t)=D_{e^{2\pi i t\ee_j}}.
$$
Then $\iota(\epsilon_j(t))=(D_{e^{2\pi i t\ee_j}},D_{e^{-2\pi i t\ee_j}})$. The first coordinate of this pair, corresponds in $\pi_1(\u_\d)$ with the sequence $\ee_j$. The second coordinate, corresponds in $\pi_1(\u_\infty(\h))$ with $-1\in\mathbb{Z}$. 
\end{proof}
\begin{coro}
$$
\pi_1(\u_{\d+\k})\simeq \ell^\infty(\mathbb{N},\mathbb{Z}) \times \mathbb{Z}\  /\  \{ (\zz, k)\in\cc_{00}^{\mathbb{Z}}\times \mathbb{Z}: -\sum_{j\ge1}z_j=k\}.
$$
\end{coro}

\subsection{ $K$-groups.}

Let us include a short comment of the $K$-groups of $\d+\k$. The extension
$$
\k \stackrel{\iota}\longrightarrow \d+\k \stackrel{\pi}\longrightarrow \ell^\infty /\cc_0
$$
induces the six term exact sequence of $K$-groups
\begin{equation}\label{Kgrupos}
\begin{array}{ccccc}
K_0(\k) & \stackrel{\iota_*}\longrightarrow & K_0(\d+\k) & \stackrel{\pi_*}\longrightarrow &  K_0(\ell^\infty)/\cc_0) \\
\partial\uparrow & \ & \ & \ & \downarrow\partial \\
K_1(\ell^\infty / \cc_0) & \stackrel{\pi_*}\longleftarrow & K_1(\d+\k) & \stackrel{\iota_*}\longleftarrow & K_1(\k) \end{array} .
\end{equation}
\begin{rem}

\noindent

\begin{enumerate}

\item
It is well known that $K_0(\k)=\mathbb{Z}$ and $K_1(\k)=0$.
\item
The $K$ groups of $\ell^\infty/\cc_0$ can be computed using standard techniques. We learned these facts from \cite{cortinas}: 
$$
K_0(\ell^\infty /\cc_0)=\ell^\infty(\mathbb{N},\mathbb{Z}) / \cc_{00}^{\mathbb{Z}},
$$
and $K_1(\ell^\infty / \cc_0)=0$.
\end{enumerate}
\end{rem}
Thus, 
\begin{coro}
$K_1(\d+\k)=0$. 
The group $K_0(\d+\k)$ contains a copy of $\mathbb{Z}$, and
$$
K_0(\d+\k) / \mathbb{Z} \simeq \ell^\infty(\mathbb{N},\mathbb{Z}) / \cc_{00}^{\mathbb{Z}}.
$$
\end{coro}
\begin{proof}
 Plugging this information in (\ref{Kgrupos}), we get that $K_1(\d+\k)=0$ and that
$$
0\longrightarrow \mathbb{Z} \longrightarrow K_0(\d+\k) \longrightarrow  \ell^\infty(\mathbb{N},\mathbb{Z}) / \cc_{00}^{\mathbb{Z}}\longrightarrow 0.
$$
\end{proof}

{\sc (Esteban Andruchow)} {Instituto de Ciencias,  Universidad Nacional de Gral. Sar\-miento,
J.M. Gutierrez 1150,  (1613) Los Polvorines, Argentina and Instituto Argentino de Matem\'atica, `Alberto P. Calder\'on', CONICET, Saavedra 15 3er. piso,
(1083) Buenos Aires, Argentina.}

\noi e-mail: {\sf eandruch@ungs.edu.ar}

\bigskip

{\sc (Eduardo Chiumiento)} {Departamento de de Matem\'atica, FCE-UNLP, Calles 50 y 115, 
(1900) La Plata, Argentina  and Instituto Argentino de Matem\'atica, `Alberto P. Calder\'on', CONICET, Saavedra 15 3er. piso,
(1083) Buenos Aires, Argentina.}     
                                               
\noi e-mail: {\sf eduardo@mate.unlp.edu.ar}    

\bigskip                                   

{\sc (Alejandro Varela)} {Instituto de Ciencias,  Universidad Nacional de Gral. Sar\-miento,
J.M. Gutierrez 1150,  (1613) Los Polvorines, Argentina and Instituto Argentino de Matem\'atica, `Alberto P. Calder\'on', CONICET, Saavedra 15 3er. piso,
(1083) Buenos Aires, Argentina.}

\noi e-mail: {\sf avarela@ungs.edu.ar}

\end{document}